\documentclass{amsart}[12pt]
\usepackage{mathpazo}
\usepackage{amsmath,amssymb,amsfonts}
\usepackage{amsthm}
\usepackage{hyperref}
\usepackage{mathabx}
\usepackage{mathrsfs}

\newtheorem{theorem}{Theorem}[section]
\newtheorem{lemma}[theorem]{Lemma}
\theoremstyle{definition}
\newtheorem{definition}[theorem]{Definition}
\newtheorem{example}[theorem]{Example}

\newtheorem{corollary}[theorem]{Corollary}

\newtheorem{proposition}[theorem]{Proposition}
\newtheorem{remark}[theorem]{Remark}
\numberwithin{equation}{section}
\newcommand\numberthis{\addtocounter{equation}{1}\tag{\theequation}}
            \title{On $\phi$-Best Proximity Points and Proximal-type Algorithms in Banach Spaces}
			\author[ P. P. Behera, C. Nahak]
			{Priyanka Priyadarshini Behera and C. Nahak}
			\address{{$^{1}$} Priyanka Priyadarshini Behera and {$^{2}$} C. Nahak,
				Department of Mathematics,
				Indian Institute of Technology Kharagpur, 
				Kharagpur, India-721302}
			\email{{$^{1}$} priyanka.pb@iitkgp.ac.in, {$^{2}$} cnahak@maths.iitkgp.ac.in}

			\date{}
\begin{document}

\begin{abstract}
The purpose of this paper is to study the existence and convergence of $\phi$-best proximity points in Banach spaces. A strong convergence theorem is established by employing a shrinking projection algorithm designed to compute common $\phi$-best proximity points of a proximally weakly suppressive mapping. Finally, we address a projection scheme to solve split $\phi$-best proximity point and variational problem in a Banach space. These contributions provide new tools for solving nonlinear problems involving non-self mappings.
\end{abstract}

\maketitle
\section{Introduction} \label{sec1}
Let $B$ be a real Banach space with the norm $\Vert \cdot \Vert$,
and $B^*$ be its dual space. Let $K$ be a nontrivial, closed subset of
$B$ and $T\colon K \to K$ be a mapping. Let $F(T)$ symbolize the
fixed point set of $T$.  Throughout the paper, we will use the
following notations:

\begin{enumerate}
\item $S_B=\{u \in B : \Vert u \Vert=1 \}$ for the unit sphere.
\item $\langle \cdot , \cdot \rangle $ for the duality pair of $B$ and
$B^*$.
\item $\rightarrow$ for strong convergence, $\xrightarrow{w}$ for  weak convergence and $\xrightarrow{w^*}$ for weak$^*$ convergence.
\end{enumerate}

A Banach space $B$ is said to be strictly convex if, for any two
distinct elements $u$, $v \in S_B$, $\Vert \frac{u+v}{2} \Vert < 1$,
and uniformly convex if for any $\epsilon\in(0,2]$, there exists
$\delta>0$, so that for any two elements $u$ and $v$ in $S_B$ with
$\Vert u-v\Vert \to 0$, it follows that $\Vert \frac{u+v}{2}
\Vert\le1-\delta$. $B$ is said to have the property (KK) if the norm
and weak convergence coincide on the unit sphere $S_B$. It has been
proved that every uniformly convex Banach space is strictly convex,
reflexive and satisfies property (KK) (see \cite{VT, J}).\\
$B$ is said to be smooth if for every $u$, $v \in S_B$ and $t \in
\mathbb{R}$, the limit
\begin{align*}
\lim_{t \to 0} \frac{\Vert u+tv \Vert - \Vert u \Vert}{t}, \numberthis \label{lim999}
\end{align*}
exists. $B$ is called uniformly smooth if limit (\ref{lim999}) is
attained uniformly.

The duality mapping $J$ of $B$  is defined by
$$ Ju=\{ \zeta \in B^* : \langle u, \zeta \rangle = \Vert u \Vert^2=\Vert \zeta \Vert^2 \}, ~ \text{for every}~ u \in B.
$$
$J$ is single-valued in a smooth Banach space $B$, and coincides the
identity operator in a Hilbert space.  Among its properties, it is
worth mentioning that it is uniformly norm-to-norm continuous on
each bounded subset of a uniformly smooth Banach space $B$ (see
\cite{ADS, VT} for more details).

Recall that for a smooth Banach space $B$,  $J$ is said to be weakly
sequentially continuous if $u_n \xrightarrow{w} u$  in $B$ implies
$Ju_n \xrightarrow{w^*} Ju$ in $B^*$ (for more details see
\cite{J}).

Let $B$ be a smooth Banach space. Following Alber \cite{A}, we
define the mapping
\begin{align*}
\phi \colon B \times B \to \mathbb{R},\quad \phi(u, v)=\Vert u
\Vert^2 - 2 \langle u, Jv \rangle + \Vert v \Vert^2. \numberthis
\label{phi}
\end{align*}
Some properties of the functional $\phi$ are enumerated below (see
\cite{A, AR, KT, MT} for more details):
\begin{itemize}
\item[($\phi$1)] $(\Vert u \Vert - \Vert v \Vert)^2 \leq \phi(u, v) \leq (\Vert u \Vert + \Vert v \Vert)^2, ~ u, v \in B.$

\item[($\phi$2)] $ \phi(u, v)=\phi(u, z)+\phi(z, v) + 2 \langle u-z, Jz-Jv \rangle, ~ u, v, z \in B.$

\item[($\phi$3)] \cite[Remark 2.1]{MT} $\phi(u, v)=0$ if and only if $u=v$, for $u, v$ in a strictly convex and smooth Banach space $B$.
\end{itemize}

Let $K$ be a nontrivial, closed, and convex subset of a strictly
convex, reflexive and smooth Banach space $B$ and $u \in B$. Then
there exists a unique element $u_0 \in K$ such that
$$
\phi(u_0, u)=\inf_{v \in K} \phi(v,u).
$$
We denote this element as $\Pi_Ku$, for every $u \in B$
and refer $\Pi_K$ as the generalized projection onto $K$. To get
more insight on generalized projections, one may refer to
\cite{KT,L, MT}.

Alber {\it et al}. \cite{A1} introduced the notion of a nonextensive
mapping in a Banach space via the generalized projection operator.
Let $(M, N)$ be two nontrivial subsets of a smooth Banach space $B$.
We say that a non-self mapping $T\colon  M \to N$ is nonextensive if
\begin{align*}
\phi(Tu, Tv) \leq \phi(u, v), ~ \text{for all}~ u, v \in M. \numberthis \label{nonexpansive}
\end{align*}
Observe that, in a Hilbert space $H$, $\phi(u, v)=\Vert u - v
\Vert^2$, for $u$, $v \in H$. Then, (\ref{nonexpansive}) is
equivalent to
 \begin{align*}
 \phi(Tu, Tv) \leq \phi(u, v) \Leftrightarrow \Vert Tu - Tv \Vert \leq \Vert u - v \Vert,\,\,\mbox{for all}\,\,u,v\in M.
\end{align*}
In this case, we recognize the widely-known nonexpansive mapping.

Let $K$ be a nontrivial, closed, convex subset of a Banach space $B$ and $f: B \to B^*$ be a mapping. The variational inequality problem defined by $f$ and $K$ is given by:
$$
VIP(f, K): ~ \text{find} ~ u^* \in K ~ \text{such that} ~ \langle f(u^*), u-u^* \rangle \geq 0, ~ \text{for every} ~ u \in K.
$$
\begin{definition}
Let $B$ be a Banach space. An operator $A: B \to 2^{B^*}$ is called 
\begin{itemize}
    \item[(i)] monotone if for $u,v \in B, ~ u^* \in Au, v^* \in Av$, $\langle u-v, u^*-v^* \rangle \geq 0$,

    \item[(ii)] $\mu$-inverse strongly monotone ($\mu$-ISM) if there exists a constant $\mu >0$ such that 
$$
\langle u-v, u^*-v^* \rangle \geq \mu \Vert u^* - v^* \Vert^2.
$$
\end{itemize}
\end{definition}

The Banach fixed point theorem concerns certain self mappings of a complete metric space and it states conditions sufficient for the existence and uniqueness of a fixed point. The theorem also gives an iterative process by which we can obtain approximations to the fixed point and error bounds. Some important fields of applications of this theorem are lies in linear algebraic equations, ordinary differential equations, integral equations etc. If we extend the fixed point problem for non-self mappings, the fixed point equation may not have a solution. In such case, it is desirable to find an approximate solution such that the error is minimized. In this concern, the notion of best proximity points has came into known. 
A mapping $T\colon B \to B$ is said to possess a fixed point if the
equation $Tu=u$ has at least one solution. When considering a
generalized projection operator, which may not behave like a metric,
we define a fixed point $u\in B$ of $T$ as one which checks the
equality $\phi(u, Tu)=0$. This appoints the fixed point equation in
a strictly convex and smooth Banach space. In the case of a non-self
mapping $T$, the equation $\phi(u, Tu)=0$ may not have a solution,
resulting in $\phi(u, Tu) >0$. More precisely, let us take two
nontrivial subsets $M$, $N$ of a smooth Banach space $B$ such that
$T\colon M \to N$ and $\phi(u, Tu) >0, u \in M$. In such a case, it
is our objective to find an element $u_0 \in M$ such that $\phi(u_0,
Tu_0)$ attains its minimum value $\text{dist}_\phi(M, N)$. This
point $u_0$ is referred to as a $\phi$-best proximity point of the
mapping $T$ in $M$.

Researchers have developed and studied numerous iterative techniques
for approximating the common fixed points of relatively nonexpansive
mappings in Banach spaces. Nakajo and Takahashi \cite{NT} devised an
iterative approach for approximating fixed points of nonexpansive
mappings in a Hilbert space. Matsushita and Takahashi \cite{MT}
proposed a hybrid approach for establishing the strong convergence
theorem for a relatively nonexpansive mapping by using the
generalized projection operator in a Banach space. Later, Inoue {\it
et al}. \cite{ITZ} addressed an enhanced shrinking projection method
for relatively nonexpansive mappings in a Banach space. This
methodology is a modified version of the iterative approach provided
in a previous study \cite{TTK} for a Hilbert space. Since then,
significant improvements have been made in obtaining hybrid
projection techniques for approximating fixed points of various
types of nonexpansive mappings in both Hilbert and Banach spaces. To
gain more insight, one may consult \cite{ITZ, KT, CSW, MT, NT, TTK}.

Strong convergence theorems by hybrid methods for the best proximity point of a nonself nonexpansive mapping in Hilbert spaces were recently proposed by Jacob et al. \cite{GMV}. A shrinking projection approach was developed by Suparatulatorn et al. \cite{RS} to solve the iterative scheme in a Hilbert space, with the aim of making computation easier. Recently, Suparatulatorn et al. \cite{RWS} introduced the general Mann algorithm for nonself nonexpansive mappings and proved the convergence result in a Hilbert space. They improved upon previous results of \cite{RS} by utilizing an approximate best proximity point sequence for a mapping instead of the demi-closedness principle. All the above convergence theorems considered the P-property, which is superfluous in the context of a Hilbert space, and hybrid techniques can be complicated while solving numerically. Moreover, these results are limited to Hilbert spaces, which needed to be expanded to more general Banach spaces. Behera and Nahak \cite{PC} have recently devised a shrinking projection approach to ascertain $\phi$-best proximity points of a non-self nonextensive mapping in a Banach space and shown the strong convergence of the iteratively generated sequences under appropriate conditions. They precisely proved the following:
\begin{theorem} \label{alg-thm1} \cite{PC}
Let $(M, N)$ be a pair of nonempty, closed, convex subsets of a uniformly convex and uniformly smooth Banach space $B$ that satisfies the $\phi_p$-property. Let $T: M \to N$ be a non-self nonextensive mapping satisfying $T(M_0) \subseteq N_0$ and that $T$ has the $\phi$-proximal property. Let us consider the sequence $\{u_n\}$ generated by 
\begin{align*}
 \begin{cases}
    u_1=\Pi_{M_0} u, u \in B ~ \text{be arbitrary}, \\
    H_1=M_0, \\
    v_n=J^{-1}(\alpha_n J u_n +(1-\alpha_n) J \Pi_M Tu_n), n \in \mathbb{N},\\
    H_{n+1} = \{ z \in H_n : \phi(z, v_n) \leq \phi(z, u_n) \}, \\
    u_{n+1} = \Pi_{H_{n+1}} u, 
    \end{cases} \numberthis \label{Pi}
\end{align*}
for each $n \in \mathbb{N}$, where $\alpha_n \in [0, a]$, for some $a \in [0, 1)$. If $\text{Best}_M^{\phi}(T)$ is nonempty, then the sequence $\{u_n\}$ strongly converges to $u^*=\Pi_{\text{Best}_M^{\phi}(T)} u$.
\end{theorem} 

Motivated by the preceding research, our focus in this work is to construct a projection algorithm for finding the common $\phi$-best proximity points for a proximally weakly suppressive mapping. This enhances Theorem \ref{alg-thm1} by leaving out the $\phi_p$-property. We have also proved a weak convergence theorem for a sequence generated by a projection scheme for computing common $\phi$-best proximity points and solution to a variational problem. 

This article organizes its contents into three parts. After a brief
introduction in Section \ref{sec1}, Section \ref{sec2} outlines the
fundamental concepts. In this section, we establish the existence and convergence of $\phi$-best proximity points. We also demonstrate a result characterizing the property ($\phi$-UC) for a pair of nontrivial subsets in a Banach space. In Section \ref{sec3}, we prove the strong convergence result of the sequence generated by a shrinking projection algorithm involving a non-self proximally weakly suppressive mapping mapping to find the common $\phi$-best proximity points in the context of a uniformly convex Banach space that is also uniformly smooth. Finally, we prove a weak convergence result for a sequence generated by a projection algorithm for finding the $\phi$-best proximity point and variational problem. This extends and improves known findings in the literature.
\section{Preliminaries} \label{sec2}
Let $M$ and $N$ be two nontrivial, closed, and convex subsets of a
smooth Banach space $B$. In this study, we denote $M_0$ and $N_0$ as
follows:
\begin{align*}
& M_0=\{u \in M : \phi(u, v)=\text{dist}_\phi(M, N), ~ \text{for some} ~ v \in N \},\\
& N_0=\{v \in N : \phi(u, v)=\text{dist}_\phi(M, N), ~ \text{for some}~ u \in M \},
\end{align*}
where $\text{dist}_\phi(M, N) = \inf \{ \phi(u, v) : u \in M ~
\text{and} ~ v \in N \}$. The pair $(M_0, N_0)$ is referred to as
the $\phi$-best proximity pair associated with $(M,N)$.

Let us exemplify this concept in two special cases.

\begin{example}\rm
Let $B$ be a Hilbert space, and $M$, $N$ two nontrivial, closed, and
convex subsets of it. In this situation,
$$\text{dist}_\phi(M, N)=\inf \{\|u-v\|^2 : u \in M ~
\text{and} ~ v \in N \}.$$
\end{example}

\begin{example}\rm
Let $B$ be a smooth Banach space, $M$ a nontrivial, closed, and convex
subset of $B$, and $N=\{0\}$, the null element of $B$. Then it is
seeable that
$$\text{dist}_\phi(M, N)=\inf \{\|u\|^2 : u \in M\}.$$
\end{example}

We designate $\text{Best}_M^{\phi}(T)$ as the set of $\phi$-best
proximity points of $T$ on $M$, where
$$ \text{Best}_M^{\phi}(T)=\{u \in M : \phi(u, Tu)=\text{dist}_\phi(M, N) \}.
$$
It is easy to see that the set $\text{Best}_M^{\phi}(T)$ is contained in $M_0$.

The following findings are essential for exhibiting the main theorems that will be presented in subsequent sections.

\begin{lemma} [Kamimura and Takahashi \cite{KT}] \label{imp1}
Let $B$ be a uniformly convex, and smooth Banach space, and let
$\{u_n\}$ and $\{v_n\}$ be two sequences in $B$. If $\phi(u_n, v_n)
\to 0$, and either $\{u_n\}$ or $\{v_n\}$ is bounded, then $\|u_n -
v_n \| \to 0$ as $n \to \infty$.
\end{lemma}

\begin{lemma} [Alber \cite{A}, Alber and Reich \cite{AR}, Kamimura and Takahashi \cite{KT}] \label{imp2}
Let $K$ be a nontrivial, closed, and convex subset of a smooth Banach
space $B$ and $u \in B$. Then $u_0=\Pi_K u$ if and only if $$
\langle u_0 - v, Ju - Ju_0 \rangle  \geq 0,~ \text{for all} ~ v \in
K.$$
\end{lemma}

\begin{lemma} [Alber \cite{A}, Kamimura and Takahashi \cite{KT}] \label{imp3}
Let $B$ be a reflexive, strictly convex, and smooth Banach space and
let $K$ be a nontrivial, closed, and convex subset of $B$ and $u \in
B$. Then
$$ \phi(v, \Pi_K u) + \phi(\Pi_K u, u) \leq \phi(v, u),\,\, \mbox{for all}\,\, v \in K. $$
\end{lemma}

\begin{lemma} [Xu \cite{X}, Z\'alinescu \cite{Z}] \label{imp4}
Let $B$ be a uniformly convex Banach space, and let $r > 0$. Then
there exists a strictly increasing, continuous and convex function
$\psi\colon [0,2r] \to [0, \infty)$ with $\psi(0)=0$ and
\begin{align*}
\Vert \alpha u + (1-\alpha) v \Vert^2 \leq \alpha \Vert u \Vert^2 +(1-\alpha) \Vert v \Vert^2 - \alpha (1-\alpha) \psi (\Vert u - v \Vert), \numberthis \label{n1}
\end{align*}
for all $u$, $v \in B_r=\{ z \in B : \Vert z \Vert \leq r \}$ and
$\alpha \in [0, 1]$.
\end{lemma}

\begin{lemma} [Kamimura and Takahashi \cite{KT}] \label{imp5}
Let $r > 0$ and $B$ be a smooth and uniformly convex Banach space.
Then, there exists a strictly increasing, continuous and convex
function $ \psi \colon [0,2r] \to [0, \infty)$ with $\psi(0)=0$, and
\begin{align*}
\psi(\Vert u-v \Vert) \leq \phi(u, v), \numberthis  \label{n2}
\end{align*}
for all $u$, $v \in B_r$.
\end{lemma}

\begin{definition} [$\phi_p$-property] \cite{PC}
Let $(M, N)$ be a pair of non-empty subsets of a smooth Banach space
$B$, with $M_0$ non-empty. Then $(M, N)$ is said to have the
$\phi_p$-property if and only if the implication
$$
 \left.\begin{aligned}
  \phi(u_1, v_1)=\text{dist}_\phi(M, N) \\
  \phi(u_2, v_2)=\text{dist}_\phi(M, N)\\ 
 \end{aligned}\right\}\Rightarrow \phi(u_1, u_2)=\phi(v_1, v_2),
$$
holds  for $u_1$, $u_2 \in M_0$ and $v_1$, $v_2 \in N_0$.
\end{definition}

We now state the following results exhibiting the existence of $\phi$-best proximity pairs $(M_0, N_0)$.

\begin{proposition}
Let $(M, N)$ be a pair of nontrivial, compact and convex subsets of a Fr\'echet smooth Banach space $B$. Then the pair $(M_0, N_0)$ is nontrivial. 
\end{proposition}
\begin{proof}
Let $\{u_n\}$ and $\{v_n\}$ be two sequences in $M$ and $N$, respectively such that $\phi(u_n, v_n)=\text{dist}_\phi(M, N)$. Now, as $(M, N)$ is a compact pair, there exists subsequence $\{u_{n_k}\}, \{v_{n_k}\}$ of $\{u_n\}, \{v_n\}$, respectively, such that $u_{n_k} \to u_0$ and $v_{n_k} \to v_0$. Then, 
\begin{align*}
\phi(u_0, v_0) & = \phi(u_0, u_{n_k})+\phi(u_{n_k}, v_{n_k})+\phi(v_{n_k}, v_0)+2 \langle u_0-u_{n_k}, Ju_{n_k}-Jv_0 \rangle \\
& \quad + 2 \langle u_{n_k}-v_{n_k}, Jv_{n_k}-Jv_0 \rangle .
\end{align*}
Using the fact that $B$ is Fr\'echet smooth, by taking limit $n$ tending to infinity, we deduce that $\phi(u_0, v_0)=\text{dist}_\phi(M, N)$. This proves the non-trivialness of the pair $(M_0, N_0)$.
\end{proof}
\begin{corollary}
Let $(M, N)$ be a pair of nontrivial, closed, convex subsets of a reflexive, Fr\'echet smooth Banach space $B$ with $M$ being bounded and $N$ being compact. Then the pair $(M_0, N_0)$ is nontrivial. 
\end{corollary}

\begin{lemma} \label{lem*}
 Let $(M, N)$ be a pair of nontrivial, closed, and convex
subsets of a smooth Banach space $B$. Then $\phi(\Pi_M v, v) = {\rm
dist}_\phi(M, N)$, for all $v \in N_0 $.
\end{lemma}

\begin{proof}
 Let $v \in N_0$. Then we can guarantee that there exists an element
$z \in M$ such that $\phi(z, v)=\text{dist}_\phi(M, N)$. Using the
definition of $\Pi_M v$, we ascertain that
\begin{align*}
\text{dist}_\phi(M, N) \leq \phi(\Pi_M v, v) = \inf_{\hat{v} \in M} \phi(\hat{v}, v) \leq \phi(z, v) = \text{dist}_\phi(M, N).
\end{align*}
Thus, $\phi(\Pi_M v, v) = \text{dist}_\phi(M, N)$.
\end{proof}

\begin{proposition}
Let $(M, N)$ be a pair of nontrivial subsets of a Fr\'echet smooth Banach space $B$ with $N$ being compact and $M$ being $\phi$-approximatively compact with respect to $N$. Then the pair $(M_0, N_0)$ is nontrivial.
\end{proposition}
\begin{proof}
Let $\{u_n\}$ and $\{v_n\}$ be two sequences in $M$ and $N$, respectively such that $\phi(u_n, v_n)=\text{dist}_\phi(M, N)$. Since $N$ is compact, by passing to a subsequence, we find $\{v_{n_k}\}$ of $\{v_n\}$ such that $v_{n_k} \to v_0$. Then
\begin{align*}
\lim_{n \to \infty} \phi(u_{n_k}, v_0) & = \lim_{n \to \infty} \phi(u_{n_k}, v_{n_k}) + \phi(v_{n_k}, v_0)+2 \langle u_{n_k}-v_{n_k}, Jv_{n_k}-Jv_0 \rangle \\
& = \text{dist}_\phi(M, N).
\end{align*} 
So, $\{u_{n_k}\}$ is a $\phi$-minimizing sequence with respect to $N$ and since $M$ is $\phi$-approximatively compact with respect to $N$, we deduce that $\{u_{n_k}\}$ has a convergent subsequence, say itself such that $u_{n_k} \to u_0$. Therefore,
\begin{align*}
\phi(u_0, v_0) & = \phi(u_0, u_{n_k})+\phi(u_{n_k}, v_{n_k})+\phi(v_{n_k}, v_0)+2 \langle u_0-u_{n_k}, Ju_{n_k}-Jv_0 \rangle \\
& \quad + 2 \langle u_{n_k}-v_{n_k}, Jv_{n_k}-Jv_0 \rangle .
\end{align*}
Using the fact that $B$ is Fr\'echet smooth and taking limit, we obtain that
$$ \phi(u_0,v_0)=\text{dist}_\phi(M, N).$$
This proves that the pair $(M_0, N_0)$ is nontrivial.
\end{proof}
\begin{corollary}
Let $(M, N)$ be a pair of nontrivial, closed, convex subsets of a Fr\'echet smooth Banach space $B$ with $M$ being locally compact and $N$ being compact. Then the pair $(M_0,  N_0)$ is nontrivial.
\end{corollary}

\begin{proposition}
Let $(M, N)$ be a pair of nontrivial, closed, convex subsets of a smooth and strictly convex Banach space $B$. Suppose $M$ is $\phi$-approximatively compact set with respect to $N$. Let $u_n$ be a sequence in $M$ and $v \in N$ satisfying $\displaystyle \lim_{n \to \infty}\phi(u_n, v) = $ dist$_\phi(M, N)$. Then $u_n$ converges to $\Pi_M(v)$.
\end{proposition}
\begin{proof}
Let $u_0 = \Pi_M(v)$. Suppose $u_n \not \to u_0$. Then there exists $\varepsilon > 0$ such that for every $n_0 \in \mathbb{N}$, $\Vert u_n - u_0 \Vert \geq \varepsilon, n \geq n_0.$ By passing to a subsequence, we find a sequence $(u_{n_j})$ of $(u_n)$ such that $$\lim_{j \to \infty} \phi(u_{n_j}, v)= \text{dist}_\phi(M, N),
$$ 
for $v \in N$. Since, $M$ is $\phi$-approximatively compact with respect to $N$, we can find a subsequence $(u_{n_{j_i}}) \subseteq (u_{n_j})$ such that $\displaystyle \lim_{i \to \infty}u_{n_{j_i}} = u$, for some $u \in M$. Then, 
\begin{align*}
\phi(u, v) & = \Vert u \Vert^2 - 2 \langle u, Jv \rangle + \Vert v \Vert^2 = \lim_{i \to \infty} \phi(u_{n_{j_i}}, v) = \text{dist}_\phi(M, N).
\end{align*}
By the uniqueness of $\Pi_M$, we have $u=u_0=\Pi_M(v)$. This contradicts to the fact that $$0 < \varepsilon \leq \Vert u_{n_{j_i}} - u_0 \Vert \to \Vert u - u_0 \Vert=0, ~ \text{as} ~ i \to \infty.$$  Thus, $\displaystyle \lim_{n \to \infty}u_n = u_0=\Pi_M(v)$.

\end{proof}

We now exhibit the closedness property of $\phi$-best proximity points of a nonextensive mapping on $M$.
\begin{theorem}
Let $(M, N)$ be a pair of nontrivial, closed, convex subsets of a Fr\'chet smooth Banach space $B$ and $T:M \to N$ be a nonextensive mapping. Then the set of $\phi$-best proximity points of $T$ on $M$ is closed.
\end{theorem}
\begin{proof}
Let $\{u_n\} \in \text{Best}_M^{\phi}(T)$ be such that $u_n \to u_0 \in M$. We need to show that $u_0 \in \text{Best}_M^{\phi}(T)$. Now,
\begin{align*}
\phi(u_0, Tu_0) & = \lim_{n \to \infty} \phi(u_n, Tu_0) \\
& = \lim_{n \to \infty} \phi(u_n, Tu_n) + 2 \langle u_n, JTu_n - JTu_0 \rangle + \Vert T u_0 \Vert^2 - \Vert Tu_n \Vert^2 \\
& \leq \lim_{n \to \infty} \phi(u_n, Tu_n) + 2 \langle u_n, JTu_n - JTu_0 \rangle + ( \Vert Tu_0 \Vert + \Vert Tu_n \Vert)(\Vert Tu_0 - Tu_n \Vert). \numberthis \label{111}
\end{align*}
We find that $\phi(Tu_n, Tu_0) \leq \phi(u_n, u_0) \to 0$ as $n$ tends to infinity. Hence, by Lemma \ref{imp1}, $\lim_{n \to \infty} \Vert Tu_n - Tu_0 \Vert =0$. Next, by Fr\'chet smoothness of $B$, one has $JTu_n - JTu_0 \to 0$ as $n$ tends to infinity. So, (\ref{111}) becomes
\begin{align*}
\phi(u_0, Tu_0) = \lim_{n \to \infty} \phi(u_n, Tu_n) = \text{dist}_\phi(M, N),
\end{align*}
which shows that $u_0 \in \text{Best}_M^{\phi}(T)$.
\end{proof}
In \cite{AP}, the property (UC) has been introduced and used to prove the existence result of best proximal points in a Banach space. Now, using the functional $\phi$, Behera and Nahak \cite{PC} introduced the property $(\phi$-UC) for a smooth Banach space. 

\begin{definition} \cite{PC}
 A pair of non-empty subsets $(M, N)$ of a smooth Banach space $B$ is said to satisfy the property $(\phi$-UC) if whenever $(u_n), (v_n)$ in $M$ and $(t_n)$ in $N$ satisfying $\phi(u_n, t_n) \to \text{dist}_\phi(M, N)$ and $\phi(v_n, t_n) \to \text{dist}_\phi(M, N)$ as $n \to \infty$, we have $\phi(u_n, v_n) \to 0$.
\end{definition}

\begin{remark}
\begin{enumerate}
    \item It may be seen that every Hilbert space has property $(\phi$-UC). This follows from the fact that in a Hilbert space $\phi(u, v)=\Vert u-v \Vert^2$. So, if we consider $\phi(u_n, t_n) \to \text{dist}_\phi(M, N)$ and $\phi(v_n, t_n) \to \text{dist}_\phi(M, N)$, for $(u_n), (v_n)$ in $M$ and $(t_n)$ in $N$, we get $\Vert u_n-t_n \Vert^2 \to \text{dist}_\phi(M, N)$ and $\Vert v_n-t_n \Vert^2 \to \text{dist}_\phi(M, N)$ as $n \to \infty$. Then $\Vert u_n-v_n \Vert \to 0$ follows by the triangle inequality. 

    \item Let $(M, N)$ be a pair of nontrivial, closed and convex subsets of a smooth and uniformly convex Banach space $B$ with $N$ being bounded. If $\phi(u_n, t_n) \to 0$ and $\phi(v_n, t_n) \to 0$, then $\Vert u_n - v_n \Vert \to 0$, for $(u_n), (v_n) \subseteq M$ and $(t_n) \subseteq N$; which follows easily by Lemma \ref{imp1}.
\end{enumerate}
\end{remark}

The following example illustrates a Banach space that does not satisfy the property ($\phi$-UC).
\begin{example}
    Let us consider $(\mathbb{R}^3, \Vert \cdot \Vert_3)$, where the norm is defined by
\begin{align*}
\Vert x \Vert_3=( \vert x_1 \vert^3 + \vert x_2 \vert^3 + \vert x_3 \vert^3)^{\frac{1}{3}}, ~ x= (x_1, x_2, x_3) \in \mathbb{R}^3.
\end{align*}
Let us take \begin{align*}
M = \{(0, x_2, x_3) : x_2^3 + x_3^3=1 \}~ \text{and} ~ N = \{ (x_1, 0, 0): 2 \leq x_1 \leq 3 \}. 
\end{align*}
Then both $M$ and $N$ are compact subsets of $\mathbb{R}^3$. Here $\text{dist}_\phi(M, N) =5$. Let us consider $$u_n=\bigg (0, \frac{n}{(n^3+1)^{\frac{1}{3}}}, \frac{1}{(n^3+1)^{\frac{1}{3}}} \bigg ), v_n= \bigg (0, \frac{1}{(n^3+1)^{\frac{1}{3}}}, \frac{n}{(n^3+1)^{\frac{1}{3}}} \bigg ) \in M, $$
and $$t_n= \bigg (2+\frac{1}{n}, 0, 0 \bigg ) \in N.$$
Then, $\phi(u_n, t_n) \to \text{dist}_\phi(M, N)$ and $\phi(v_n, t_n) \to \text{dist}_\phi(M, N)$; but $\lim_{n\to \infty} \phi(u_n, v_n)=2$. This shows that $(M, N)$ does not satisfy the property ($\phi$-UC).

\end{example}

The following theorem provides a sufficient condition for a smooth Banach space to satisfy the property ($\phi$-UC).
\begin{theorem}
Let $B$ be a reflexive, strictly convex, Fr\'echet smooth Banach space with property (KK). Let $M$ be a non-empty, closed, convex subset and $N$ be a non-empty, compact subset of $B$. Then, $(M, N)$ satisfies the property ($\phi$-UC).
\end{theorem}
\begin{proof}
Let us assume dist$_\phi(M, N) > 0$. Let $(u_n), (v_n) \subseteq M$ and $(t_n) \subseteq N$ be such that \
$$
\lim_{n \to \infty} \phi(u_n, t_n) = \text{dist}_\phi(M, N) \hspace{0.5cm} \text{and} \hspace{0.5cm} \lim_{n \to \infty} \phi(v_n, t_n) = \text{dist}_\phi(M, N).
$$ 
Since $N$ is compact, $(t_n)$ has a convergent subsequence, say $(t_{n_k}) \subseteq (t_n)$ such that $t_{n_k} \to t_0$, for some element $t_0 \in N$. Then, from the inequality, 
$$
0 \leq ( \Vert u_{n_k} \Vert - \Vert t_{n_k} \Vert )^2 \leq \phi(u_{n_k}, t_{n_k}) \leq \phi(u_{n_k} , t_0) \leq (\Vert u_{n_k} \Vert + \Vert t_0 \Vert)^2,
$$
and $t_{n_k} \to t_0$, we get that $(u_{n_k})$ is a bounded subset of $M$. Similarly, we find that $(v_{n_k})$ is a bounded subset of $N$. Using reflexivity criteria of $B$, we can find subsequences of $(u_{n_k})$ and $(v_{n_k})$, say themselves such that $u_{n_k} \xrightarrow{w} u_0$ and $v_{n_k} \xrightarrow{w} v_0$, for some elements $u_0, v_0 \in M$. Now, we need to show that $\Vert u_{n_k} \Vert \to \Vert u_0 \Vert$ as $k \to \infty$.\\
\\
From the convexity of $M$ and by the inequality, 
$$
\phi(u_0, t_0) \leq \phi \bigg ( \frac{u_0+u_{n_k}}{2} , t_0 \bigg ),
$$ we get
$$
\Vert u_0 \Vert^2 - 2 \langle  u_0, J t_0 \rangle + \Vert t_0 \Vert^2 \leq \bigg \Vert \frac{u_0+u_{n_k}}{2} \bigg \Vert^2 - 2 \bigg \langle  \frac{u_0+u_{n_k}}{2}, J t_0 \bigg \rangle + \Vert t_0 \Vert^2.
$$
This simplifies to,
\begin{align*}
\Rightarrow 2 \bigg \langle  \frac{u_{n_k} - u_0}{2}, Jt_0 \bigg \rangle \leq \bigg \Vert \frac{u_0+u_{n_k}}{2} \bigg \Vert^2 - \Vert u_0 \Vert^2. \numberthis \label{eqn1}
\end{align*}
Likewise, using the inequality $\phi(u_{n_k}, t_{n_k}) \leq \phi(u_0, t_{n_k})$, we get
\begin{align*}
2 \langle  u_{n_k}-u_0, -J t_{n_k} \rangle \leq \Vert u_0 \Vert^2 - \Vert u_{n_k} \Vert^2. \numberthis \label{eqn2}
\end{align*}
Adding (\ref{eqn1}) and (\ref{eqn2}), we get
\begin{align*}
2 \langle  u_{n_k}-u_0, J t_0 - J t_{n_k} \rangle - \bigg \langle \frac{u_{n_k}-u_0}{2}, J t_0 \bigg \rangle \leq \bigg  \Vert \frac{u_0+u_{n_k}}{2} \bigg \Vert^2 - \Vert u_{n_k} \Vert^2.
\end{align*}
This further reduces to
\begin{align*}
\Rightarrow 2  \langle  u_0-u_{n_k}, Jt_0 - J t_{n_k} \rangle \geq \frac{1}{2} ( \Vert u_{n_k} \Vert^2 - \Vert u_0 \Vert^2) + \langle  u_0-u_{n_k}, Jt_0 \rangle. \numberthis \label{eqn3}
\end{align*}
Similarly, using the inequalities $$\phi(u_0, t_0) \leq \phi(u_{n_k}, t_0) ~ \text{and} ~ \phi(u_{n_k}, t_{n_k}) \leq \phi \bigg (\frac{u_0+u_{n_k}}{2}, t_{n_k} \bigg ),$$ we obtain
\begin{align*}
2 \langle u_0-u_{n_k}, Jt_0 - J t_{n_k} \rangle \geq \frac{1}{2} ( \Vert  u_0  \Vert^2 - \Vert u_{n_k}\Vert^2) + \langle u_0-u_{n_k}, Jt_{n_k} \rangle. \numberthis \label{eqn4}
\end{align*}
Since, $t_{n_k} \to t_0$ and $u_{n_k} \xrightarrow{w} u_0$, (\ref{eqn3}) and (\ref{eqn4}) gives,
\[
\Vert u_{n_k} \Vert \to \Vert u_0 \Vert. 
\]
We can similarly get $\Vert v_{n_k} \Vert \to \Vert v_0 \Vert$. Hence, by the property (KK), $u_{n_k} \to u_0$ and $v_{n_k} \to u_0$ as $k \to \infty$. \\
\\
Further, 
\begin{align*}
  \phi(u_0, t_0) & = \lim_{k \to \infty} \phi(u_{n_k}, t_0) \\
  & = \lim_{k \to \infty} \phi(u_{n_k}, t_{n_k}) + \phi(t_{n_k}, t_0 ) + 2 \langle u_{n_k} - t_{n_k} , J t_{n_k} - J t_0 \rangle \\
  & = \text{dist}_\phi(M, N),
\end{align*} 
which is possible because of the fact that $B$ is Fr\'echet smooth and $t_{n_k} \to t_0$ as $k \to \infty$. Similarly, we also get $$\phi(v_0, t_0)=\text{dist}_\phi(M, N).$$
So, $\phi(u_0, t_0) = \phi(v_0, t_0) =$ dist$_\phi(M, N)$. Strict convexity implies, $u_0=v_0$ and so, 
$$
\lim_{k \to \infty} \phi(u_{n_k}, v_{n_k}) = \phi(u_0, v_0) = 0.
$$
Thus, the result follows.

\end{proof}

An easy consequence of the above theorem is stated below.
\begin{corollary}
Let $M$ and $N$ be two non-empty, closed, convex subsets of a Fr\'echet smooth and uniformly convex Banach space with $N$ being compact. Then $(M, N)$ satisfies the property ($\phi$-UC).
\end{corollary}

\begin{theorem}
Let $M$ be a convex, weakly compact, and $N$ be a compact subset of a Fr\'echet smooth and strictly convex Banach space $B$. Then $(M, N)$ satisfies the property ($\phi$-UC).
\end{theorem}
\begin{proof}
Assume that $\text{dist}_\phi(M, N) > 0$. Let $(u_n), (v_n) \subseteq M$ and $(t_n) \subseteq N$ be such that $$
\lim_{n \to \infty} \phi(u_n, t_n) = \text{dist}_\phi(M, N) \hspace{0.5cm} \text{and} \hspace{0.5cm} \lim_{n \to \infty} \phi(v_n, t_n) = \text{dist}_\phi(M, N).
$$ 
Since $M$ is weakly compact and $N$ is compact, there exists sequences $(u_{n_j}) \subseteq (u_n)$, $(v_{n_j}) \subseteq (v_n)$ and $(t_{n_j}) \subseteq (t_n)$ such that $u_{n_j} \xrightarrow{w} u_0$, $v_{n_j} \xrightarrow{w} v_0$, for some elements $u_0, v_0 \in M$ and $t_{n_j} \to t_0$, for some element $t_0 \in N$, respectively. Then, 
\begin{align*}
\phi(u_0, t_0) & = \Vert u_0 \Vert^2 - 2 \langle  u_0, Jt \rangle + \Vert t_0 \Vert^2\\
& \leq \liminf_{j \to \infty} (\Vert u_{n_j} \Vert^2 - 2 \langle  u_{n_j},J t_{n_j} \rangle + \Vert t_{n_j} \Vert^2  + 2 \langle u_{n_j}, Jt_{n_j}-Jt_0 \rangle - \Vert t_{n_j} \Vert^2+\Vert t_0 \Vert^2) \\
& \leq \lim_{j \to \infty} (\phi(u_{n_j}, t_{n_j}) + 2 \langle  u_{n_j}, Jt_{n_j}-Jt_0 \rangle + (\Vert t_{n_j} \Vert + \Vert t_0 \Vert) \Vert t_{n_j} - t_0 \Vert). \numberthis \label{2}
\end{align*}
Here, by Fr\'echet smoothness of $B$, $\langle u_{n_j}, Jt_{n_j}-Jt_0 \rangle \to 0$ as $j \to \infty$. So, (\ref{2}) reduces to $\phi(u_0, t_0)=\text{dist}_\phi(M, N)$. In a similar way, we also obtain $\phi(v_0, t_0) =  \text{dist}_\phi(M, N)$. Strictly convexity gives $u_0 = v_0$. This results, $$\lim_{j \to \infty} \phi(u_{n_j}, v_{n_j}) = \phi(u_0, v_0) = 0.$$ Since, $(n_j)$ is arbitrary, $\phi(u_n, v_n) \to 0$ as $n \to \infty$. Hence, $(M, N)$ satisfies the property ($\phi$-UC).

\end{proof}

\section{Convergence Analysis} \label{sec3}
 In this section, we establish a strong convergence theorem by introducing the shrinking projection algorithm, which is used to find a common solution of $\phi$-best proximity point set of a non-self proximally weakly suppressive mapping in a Banach space, which is uniformly convex and uniformly smooth. Along the section, we are going to assume that the set $M_0$ is convex. Let us first introduce the following notions required to prove the main conclusion.
 \begin{definition}  [A$\phi$-BPS] \cite{PC}
 Let $(M, N)$ be a nontrivial pair of subsets of a smooth Banach space $B$ and
 $T\colon M \to N$ be a non-self mapping. A sequence $\{u_n\}$ in $M$ is said to be an approximate
 $\phi$-best proximity point sequence (A$\phi$-BPS) for $T$ if and only if $\displaystyle \lim_{n \to \infty} \phi(u_n, Tu_n) = \text{dist}_\phi(M, N)$.
 \end{definition}

Note that the above definition is a modified version of an
approximate best-proximity point sequence as studied in \cite{G} for a metric space. The same paper examined the notion of proximal
property for non-self mapping in a normed linear space. We define it as follows: for two nontrivial subsets $M$, $N$ of a normed linear space $B$, a non-self mapping $T\colon  M \to N$ is said to have the proximal property if for each sequence $\{u_n\}$ in $M$ with $u_n \xrightarrow{w} u \in M$ and $\Vert u_n - Tu_n \Vert \to d(M, N)$, $\Vert u - Tu \Vert = d(M, N)$, where $d(M, N)=\inf \{d(u, v) : u \in M, v \in N \}$. Employing this property, Suparatulatorn {\it et al}. \cite{RWS} obtained a strong convergence result of a new hybrid scheme for best proximity points in a Hilbert space. Here, we give a modified definition of the proximal property via the generalized projection operator in a smooth Banach space.

  \begin{definition} [$\phi$-proximal property] \cite{PC}
  Let $(M, N)$ be a pair of nontrivial subsets of a smooth Banach space $B$.
  A non-self mapping $T\colon M \to N$ is said to satisfy the $\phi$-proximal
  property if and only if for each sequence $\{u_n\}$ in $M$ such that
  $u_n \xrightarrow{w} u_0 \in M$ and $\displaystyle \lim_{n \to \infty} \phi(u_n, Tu_n)=\text{dist}_\phi(M, N)$, we have $\phi(u_0, Tu_0)=\text{dist}_\phi(M, N)$.
  \end{definition}

  If $\text{dist}_\phi(M, N) = 0$ and $B$ is smooth, and strictly convex, then the $\phi$-proximal property
  reduces to the demi-closedness principle of $I - T$ at 0, where $I$ is the identity operator on $M$.
  Recall that the mapping $I-T \colon  M \to B$ is demi-closed at 0 if whenever $\{u_n\}$ is a
  sequence in $M$ such that $u_n \xrightarrow{w} u_0 \in M$ and $(I-T) u_n \to 0$ as $n \to \infty$, we have $(I-T)u_0 = 0$.
\begin{definition} \label{d}
Let $(M, N)$ be a non-empty pair of subsets of a smooth Banach space $B$. Then, $T: M \to N$ is said to be
\begin{enumerate}
\item \cite{A1} {\it weakly suppressive mapping} if 
$$
\phi(Tu_1, Tu_2) \leq \phi(u_1, u_2) - \eta(\phi(u_1, u_2)), \forall u_1, u_2 \in M,
$$ 
where $\eta : [0, \infty) \to [0, \infty)$ is a continuous and non-decreasing function such that $\eta$ is positive on $(0, \infty)$, $\eta(0)=0$ and $\lim_{t \to \infty} \eta(t)=\infty$.

\item { \it proximally weakly suppressive mapping} if and only if 
$$
 \left.\begin{aligned}
   \phi(u_1, Tv_1) &=\text{dist}_\phi(M, N) \\
   \phi(u_2, Tv_2) &=\text{dist}_\phi(M, N)
 \end{aligned}\right\}\Rightarrow \phi(u_1, u_2) \leq \phi(v_1, v_2) - \eta(\phi(v_1, v_2)),
$$
where $u_1, u_2, v_1, v_2 \in M$.
\end{enumerate}   
\end{definition}
In a Hilbert space, the notions of weakly suppressive and proximally weakly suppressive mappings generalized into weakly contractive and proximally weakly contractive mappings, respectively (see \cite{V, GM} for more details).

Let us now provide a shrinking projection scheme for finding the $\phi$-best proximity points in the setting of a uniformly smooth Banach space that is also uniformly convex.

\begin{theorem} \label{main}
Let $(M, N)$ be a pair of nontrivial, closed, convex subsets of a uniformly convex and uniformly smooth Banach space $B$. Let $T: M \to N$ be a proximally weakly suppressive mapping satisfying $T(M_0) \subseteq N_0$ and $\text{Best}_M^{\phi}(T)$ is nontrivial. Let us consider the sequence $\{u_n\}$ generated by 
\begin{align*}
 \begin{cases}
    u_1=\Pi_{M_0} u_0, u_0 \in M_0 ~ \text{be arbitrary}, \\
    H_1=M_0, \\
    v_n=J^{-1}(\alpha_n J u_n +(1-\alpha_n) J \Pi_M Tu_n), n \in \mathbb{N},\\
    H_{n+1} = \{ z \in H_n : \phi(z, v_n) \leq \phi(z, u_n) \}, \\
    u_{n+1} = \Pi_{H_{n+1}} u_0, 
    \end{cases} \numberthis \label{Piw}
\end{align*}
for each $n \in \mathbb{N}$, where $\alpha_n \in [0, a]$, for some $a \in [0, 1)$. Then, if $T$ satisfies the $\phi$-proximal property, $\{u_n\}$ strongly converges to $u^*=\Pi_{\text{Best}_M^\phi(T)} u_0 $.
\end{theorem}
\begin{proof}
Let us divide the proof into four steps.

{\sc Step 1}: In the first step, we prove that $H_n$ is closed and
convex, for all $n \in \mathbb{N} \cup \{0\}$. Clearly, $H_1=M_0$,
which is closed and convex. Assume that $H_k$ is closed and convex,
for some $k \in \mathbb{N}$. Then, for all $z \in H_{k+1}$, we have
$\phi(z, v_k) \leq \phi(z, u_k)$, which is equivalent to
$$
2 \langle z, Ju_k - Jv_k \rangle \leq \Vert u_k \Vert^2 - \Vert x_k \Vert^2.
$$
This characterization allows us the conclusion that  $H_{k+1}$ is
closed and convex, and hence for all $n \in \mathbb{N}$, $H_n$ is
closed and convex. This also shows that $\Pi_{H_{n+1}} u$ is
well-defined.\\
 Next, we show that $ \text{Best}_M^{\phi}(T) \subset H_n$, for all $n \in
\mathbb{N}$, which also implies that $H_n$ is  not void. For $n=1$,
$\text{Best}_M^{\phi}(T) \subset H_1 = M_0$. Let us assume that $\text{Best}_M^{\phi}(T) \subset H_k$, for some $k \in \mathbb{N}$. Then, for $p \in \text{Best}_M^\phi(T)$, we have
\begin{align*}
\phi(p, v_k) & = \phi(p, J^{-1}(\alpha_k J u_k +(1-\alpha_k) J \Pi_M Tu_k)) \\
& \leq \Vert p \Vert^2 - 2 \alpha_k \langle p, u_k \rangle - 2(1-\alpha_k) \langle p, \Pi_{M_0} Tu_k \rangle + \alpha_k \Vert u_k \Vert^2 +(1-\alpha_k) \Vert \Pi_M Tu_k \Vert^2 \\
& = \alpha_k \phi(p, u_k) +(1-\alpha_k) \phi(p, \Pi_M Tu_k). \numberthis \label{s1}
\end{align*}
Now, we have $\phi(p, Tp)=\text{dist}_\phi(M, N)$ and $\phi(\Pi_M Tu_k, Tu_k)=\text{dist}_\phi(M, N)$. By using the fact that $T$ is proximally weakly suppressive mapping, we conclude that
\begin{align*}
\phi(p, \Pi_M Tu_k) \leq \phi(p, u_k) - \eta \phi(p, u_k). \numberthis \label{s2}
\end{align*}
Substituting (\ref{s2}) in (\ref{s1}), we deduce that
\begin{align*}
\phi(p, v_k) & \leq \alpha_k \phi(p, u_k) +(1-\alpha_k) [ \phi(p, u_k)-\eta (\phi(p, u_k))] \\
& = \phi(p, u_k) - (1-\alpha_k) \eta (\phi(p, u_k)) \\
& \leq \phi(p, u_k).
\end{align*}
This implies that $p \in H_{k+1}$. Thus, $F^\phi \subset H_n$, for
all $n \in \mathbb{N}$.

{\sc Step 2}: In this stage, our goal is to show that $\lim_{n \to
\infty} \phi(u_n, u_0)$ exists.
From $u_n = \Pi_{H_n} u_0$ and by Lemma \ref{imp2}, we deduce that
$$\langle u_n - p, Ju_0 - Ju_n \rangle \geq 0,\,\, \mbox{for
all}\,\, p \in F^\phi.$$
Also, by Lemma \ref{imp3}, it follows that
\begin{align*}
\phi(u_n, u_0) = \phi(\Pi_{H_n} u_0, u_0) \leq \phi(p, u_0) -
\phi(p, u_n) \leq \phi(p, u_0),
\end{align*}
for each $p \in F^\phi \subset H_n$, for all $n \in \mathbb{N}$.
Thus, the sequence $\{ \phi(u_n, u_0) \}$ is bounded, and from the
inequality $$(\Vert u_n \Vert - \Vert u_0 \Vert)^2 \leq \phi(u_n,
u_0),$$
 it follows that $\{u_n\}$ is a bounded sequence.
From $u_{n+1}=\Pi_{H_{n+1}} u_0 \in H_{n+1} \subset H_n$, one has
\begin{align*}
\phi(u_n, u_0) \leq \phi(u_{n+1}, u_0),\,\, \mbox{for all}\,\, n \in
\mathbb{N} \cup \{0\}.
\end{align*}
As a result, $\{ \phi(u_n, u_0) \}$ is non-decreasing and bounded.
Hence the limit of $\{\phi(u_n, u_0)\}$ exists.

{\sc Step 3}: We next claim that $\lim_{n\to\infty}\phi(u_n, Tu_n) =
\text{dist}_\phi(M, N)$.
 Now, from Lemma \ref{imp3}, it follows that
\begin{align*}
\phi(u_{n+1}, u_{n}) & = \phi(u_{n+1}, \Pi_{H_{n}}u_0) \leq
\phi(u_{n+1}, u_0) - \phi(u_{n}, u_0).
\end{align*}
So, letting $n \to \infty$, we have
\begin{align*}
\lim_{n \to \infty} \phi(u_{n+1}, u_{n}) = 0.
\end{align*}
Thus, by Lemma \ref{imp1},
\begin{align*}
\lim_{n \to \infty} \Vert u_{n+1} - u_{n} \Vert = 0. \numberthis
\label{s6}
\end{align*}
Noticing that $u_{n+1} \in H_{n+1}$, we get $\phi(u_{n+1}, v_n) \leq
\phi(u_{n+1}, u_n).$ It follows that $ \lim_{n \to \infty}
\phi(u_{n+1}, v_n) =0$, and, from  Lemma \ref{imp1}, we get that
\begin{align*}
\lim_{n \to \infty} \Vert u_{n+1} - v_n \Vert = 0.\numberthis \label{s7^*}
\end{align*}
 From \eqref{s6} and \eqref{s7^*},
we have
\begin{align*}
\lim_{n \to \infty} \Vert u_n - v_n \Vert = 0. \numberthis \label{s7}
\end{align*}
Noticing that $J$ is uniformly norm-to-norm continuous on bounded sets, it follows that
\begin{align*}
\lim_{n \to \infty} \Vert Ju_n - Jv_n \Vert = 0. \numberthis \label{s8}
\end{align*}
Now, 
\begin{align*}
\phi(u_n, v_n) & = \phi(u_n, J^{-1}(\alpha_n J u_n +(1-\alpha_n) J \Pi_M Tu_n)) \\
& \leq \alpha_n \phi(u_n, u_n)+(1-\alpha_n) \phi(u_n, \Pi_M Tu_n) \\
& = (1-\alpha_n) \phi(u_n, \Pi_M Tu_n). \numberthis \label{s9}
\end{align*}
Substituting (\ref{s7}) in (\ref{s9}) and using the fact that $\limsup_{n \to \infty} \alpha_n <1$, we get
\begin{align*}
\lim_{n \to \infty} \phi(u_n, \Pi_M Tu_n)=0. \numberthis \label{s10}
\end{align*}
 Then, we get
\begin{align*}
\phi(u_n, \Pi_M Tu_n) & = \Vert u_n \Vert^2 - 2\langle u_n, J\Pi_M Tu_n \rangle + \Vert \Pi_M Tu_n \Vert^2 \\
 & = \Vert u_n \Vert^2 - \Vert \Pi_M Tu_n \Vert^2 - 2 \langle u_n - \Pi_M Tu_n, J \Pi_M Tu_n \rangle \\
 & \leq (\Vert u_n - \Pi_M Tu_n \Vert)(\Vert u_n \Vert + \Vert \Pi_M Tu_n \Vert) - 2 \langle u_n - \Pi_M Tu_n, J \Pi_M Tu_n \rangle \\
 & \leq (\Vert u_n - \Pi_M Tu_n \Vert)(\Vert u_n \Vert + \Vert \Pi_M Tu_n \Vert) + 2 \Vert u_n - \Pi_M Tu_n \Vert \Vert J \Pi_M Tu_n \Vert.
\end{align*}
 Since $\{u_n\}$ is bounded, $\{\Pi_MTu_n\}$ is bounded, and,  as $J$ is continuous norm-to-norm, $\{J \Pi_M Tu_n\}$ is bounded. Now, using
\eqref{s10}, we conclude that
\begin{align*}
 \lim_{n \to \infty} \phi(u_n, \Pi_M Tu_n)=0. \numberthis \label{s12}
\end{align*}
A straightforward computation yields
\begin{align*}
\phi(u_n, Tu_n) & = \phi(u_n, \Pi_M Tu_n) + \phi(\Pi_M Tu_n, Tu_n) + 2 \langle u_n - \Pi_M Tu_n, J \Pi_M Tu_n - J Tu_n \rangle. \numberthis \label{s13}
\end{align*}
Letting $n$ to infinity and using \eqref{s10}, equality
\eqref{s13} reduces to 
\begin{align*}
\lim_{n \to \infty} \phi(u_n, Tu_n) =\lim_{n \to \infty} \phi(\Pi_M Tu_n, Tu_n)=\text{dist}_\phi(M, N) .
\end{align*}
Thus, $\{u_n\}$ is an (A$\phi$-BPS) with respect to the mapping $T$.

{\sc Step 4}: In the last part, we prove that $\{u_n\}$ strongly
converges to $\Pi_{\text{Best}_M^{\phi}(T)}u_0$. Since $\{u_n\}$ is bounded, there
exists a weakly convergent subsequence, say $\{u_{n_k} \}$ of
$\{u_n\}$, such that $u_{n_k} \xrightarrow{w} q \in M_0$. Observing
the fact that $T$ satisfy the $\phi$-proximal property, $q
\in \text{Best}_M^{\phi}(T)$. \\
 Let
$u^*=\Pi_{\text{Best}_M^\phi (T)} u_0$. Since the norm is weakly lower
semicontinuous, we deduce that
\begin{align*}
\phi(q, u_0) & = \Vert q \Vert^2 - 2 \langle q, Ju_0 \rangle + \Vert u_0\Vert^2  \leq \liminf_{k \to \infty} \phi(u_{n_k}, u_0) \leq \limsup_{k \to \infty} \phi(u_{n_k}, u_0)  \leq \phi(u^*, u_0).
\end{align*}
From the definition of $\Pi$, we have $q=u^*$. Hence,
\begin{align*}
\lim_{k \to \infty} \phi(u_{n_k}, u_0) = \phi(u^*, u_0).
\end{align*}
So,
\begin{align*}
0=\lim_{k \to \infty} (\phi(u_{n_k}, u_0) - \phi(u^*, u_0)) = \lim_{k \to \infty} (\Vert u_{n_k} \Vert^2 - \Vert u^* \Vert^2),
\end{align*}
resulting in, $\Vert u_{n_k} \Vert \to \Vert u^* \Vert$ as $k \to \infty$.
Noticing that $B$ has property (KK), we have $u_{n_k} \to u^*$. Hence $\{u_n\}$ converges strongly to $\Pi_{\text{Best}_M^\phi (T)} u_0$. \\
By the definition of proximally weakly suppressive mapping and property of $\Pi_M$, we have for $u^* \in \text{Best}_M^\phi (T)$, 
$$ \phi(u^*, T u^*) = \text{dist}_\phi(M, N) ~ \text{and} ~ \phi(\Pi_M Tu_n, Tu_n)=\text{dist}_\phi(M, N).
$$
So, $\phi(u^*, \Pi_M Tu_n) \leq \phi(u^*, u_n) - \eta (\phi(u^*, u_n))$.
This implies that 
\begin{align*}
\eta (\phi(u^*, u_n)) & \leq \phi(u^*, u_n) - \phi(u^*, \Pi_M Tu_n) \\
& = \Vert u_n \Vert^2 - \Vert \Pi_M Tu_n \Vert^2 + 2 \langle u^*, J \Pi_M Tu_n \rangle - 2 \langle u^*, Ju_n \rangle \\
& \leq \Vert u_n - \Pi_M Tu_n \Vert (\Vert u_n \Vert + \Vert \Pi_M Tu_n \Vert)+ 2 \Vert u^* \Vert \Vert J \Pi_M Tu_n - J u_n \Vert.
\end{align*}
Since $\lim_{n \to \infty} \Vert u_n - \Pi_M Tu_n \Vert =0$ and $\{\Pi_M Tu_n \}$ is bounded, by the uniform norm-norm continuity of $J$ on bounded subsets of $B$, we obtain that 
$$ \lim_{n \to \infty} \eta (\phi(u^*, u_n)) =0.$$
By the properties of $\eta$, it follows that $\lim_{n \to \infty} \phi(u^*, u_n) =0$ and so, $\lim_{n \to \infty} \Vert u_n - u^* \Vert =0$. i.e., $u_n \to u^*$. On the account of the uniqueness of the limit of $\{u_n\}$, we find that $u^* = \Pi_{\text{Best}_M^\phi(T)} u_0$.

\end{proof}

\begin{theorem} [In Hilbert space]
Let $(M, N)$ be a pair of nonempty, closed, convex subsets of a Hilbert space $\mathbb{H}$. Let $T: M \to N$ be a proximally weakly suppressive mapping satisfying $T(M_0) \subseteq N_0$ and that $T$ has the proximal property. Let us consider the sequence $\{u_n\}$ generated by 
\begin{align*}
 \begin{cases}
    u_1=P_{M_0} u_0, u_0 \in M_0 ~ \text{be arbitrary}, \\
    H_1=M_0, \\
    v_n=\alpha_n  u_n +(1-\alpha_n)  P_M Tu_n, n \in \mathbb{N},\\
    H_{n+1} = \{ z \in H_n : \Vert z- v_n \Vert \leq \Vert z -u_n \Vert \}, \\
    u_{n+1} = P_{H_{n+1}} u_0, 
    \end{cases} \numberthis \label{Pi2}
\end{align*}
for each $n \in \mathbb{N}$, where $\alpha_n \in [0, a]$, for some $a \in [0, 1)$. If $\text{Best}_M(T)$ is nonempty, then the sequence $\{u_n\}$ strongly converges to $u^*=P_{\text{Best}_M(T)} u_0$.
\end{theorem}
\begin{proof}
In a Hilbert space, $\phi(u, v)=\Vert u-v \Vert^2$, for all $u, v \in \mathbb{H}$. So, proximally weakly suppressive mapping coincides with proximally weakly contractive mappings and rest of the proof follows from Theorem \ref{main}.
\end{proof}

\begin{remark}
Theorem \ref{main} improves upon the results of \cite{PC, GMV, RWS, RS} in the following ways:
\begin{itemize}
\item[(i)] It does not require the assumption of the $\phi_p$-property, thereby broadening the class of admissible spaces and mappings.

\item[(ii)] It generalizes existing best proximity point results from Hilbert spaces to more general Banach spaces.

\item[(iii)] It considers proximally weakly suppressive mappings instead of nonextensive mappings, thus allowing a wider class of operators that are typically involved in hybrid-type algorithms.
\end{itemize}
\end{remark}
Let us now propose a projection method to address the split $\phi$-best proximity point and variational problems in a Banach space. 

\begin{theorem} [Weak Convergence Theorem] \label{weak}
Let the pair $(M, N)$ and $K$ be nontrivial, closed, convex subsets of uniformly convex and uniformly smooth Banach spaces $B_1$ and $B_2$, respectively. Let $A: B_1 \to B_2$ be a bounded linear operator with $A^*$ be its dual and $f:K \to B_2^*$ be a $\mu$-inverse strongly monotone operator such that $VIP(f, K) \neq \emptyset$. Let $T: M \to N$ be a proximally weakly suppressive operator such that $T(M_0) \subseteq N_0$ and $\text{Best}_M^\phi (T) \neq \emptyset$. For an arbitrarily choosen $u_0 \in M_0$, let us consider the sequence $\{u_n\}$ iteratively by 
\begin{align*}
 \begin{cases}
    v_n=J_{B_1}^{-1}((1-\alpha_n) J_{B_1} u_n + \alpha_n  J_{B_1} \Pi_M Tu_n), \\
    u_{n+1} = \Pi_M J_{B_1}^{-1}(J_{B_1} v_n + \gamma A^* J_{B_2}(U-I) Av_n), ~ \forall n \geq 0, 
    \end{cases} \numberthis \label{Sp}
\end{align*}
where $U=\Pi_K (J_{B_1} - \lambda f)$ with $\lambda \in [0, 2 \mu), ~ \gamma \in (0, \frac{1}{L})$, $L$ is the spectral radius of $A^* A$, $\alpha_n \subset (0, 1]$ with $\limsup_{n \to \infty} \alpha_n <1$ and $(U-I)$ is demi-closed at zero. Suppose $\Sigma = \{ p \in \text{Best}_M^\phi (T) : x=Ap \in VIP(f, K) \} \neq \emptyset$ and $T$ satisfies the $\phi$-proximal property, then $\{u_n\}$ converges weakly to an element $u^* \in \Sigma$.
\end{theorem}
\begin{proof}
Let $p \in \Sigma$. Then $p \in \text{Best}_M^\phi(T)$. i.e., $\phi(p, Tp) = \text{dist}_\phi(M, N)$. We also know by Lemma \ref{lem*} that $\phi(\Pi_M Tu_n, Tu_n) = \text{dist}_\phi(M, N)$. Now, since $T$ is proximally weakly suppressive operator, we have
$$
\phi(p, \Pi_M Tu_n) \leq \phi(p, u_n) - \eta (\phi(p, u_n)),
$$
where $\eta$ is as described in Definition \ref{d}(ii).
Now,
\begin{align*}
\phi(p, v_n) & = \phi(p, J_{B_1}^{-1}((1-\alpha_n) J_{B_1} u_n + \alpha_n  J_{B_1} \Pi_M Tu_n)) \\
& = \Vert p \Vert^2 - 2 \langle p, (1-\alpha_n) J_{B_1} u_n + \alpha_n  J_{B_1} \Pi_M Tu_n \rangle + \Vert (1-\alpha_n) J_{B_1} u_n + \alpha_n  J_{B_1} \Pi_M Tu_n \Vert^2 \\
& \leq \Vert p \Vert^2 - 2 (1-\alpha_n) \langle p, J_{B_1}u_n \rangle - 2 \alpha_n \langle p, J_{B_1} \Pi_M Tu_n \rangle + (1-\alpha_n) \Vert u_n \Vert^2 + \alpha_n \Vert \Pi_M Tu_n \Vert^2 \\
& \quad - \alpha_n (1-\alpha_n) \psi(\Vert J_{B_1} u_n - J_{B_1} \Pi_M Tu_n \Vert) \\
& = (1-\alpha_n) \phi(p, u_n) + \alpha_n \phi(p, \Pi_M Tu_n) - \alpha_n (1-\alpha_n) \psi(\Vert J_{B_1} u_n - J_{B_1} \Pi_M Tu_n \Vert) \\
& \leq (1-\alpha_n) \phi(p, u_n) + \alpha_n [\phi(p, u_n)- \eta(\phi(p, u_n)] - \alpha_n (1-\alpha_n) \psi(\Vert J_{B_1} u_n - J_{B_1} \Pi_M Tu_n \Vert) \\
& = \phi(p, u_n) - \alpha_n \eta(\phi(p, u_n) - \alpha_n (1-\alpha_n) \psi(\Vert J_{B_1} u_n - J_{B_1} \Pi_M Tu_n \Vert). \numberthis \label{sp1}
\end{align*}
Next, 
\begin{align*}
\phi(p, u_{n+1}) & = \phi(p, \Pi_M J_{B_1}^{-1}(J_{B_1} v_n + \gamma A^* J_{B_2}(U-I) Av_n)) \\
& \leq \phi(p, J_{B_1}^{-1}(J_{B_1} v_n + \gamma A^* J_{B_2}(U-I) Av_n)) \\
& = \Vert p \Vert^2 - 2 \langle p, J_{B_1} v_n + \gamma A^* J_{B_2}(U-I) Av_n \rangle + \Vert J_{B_1} v_n + \gamma A^* J_{B_2}(U-I) Av_n \Vert^2 \\
& \leq  \Vert p \Vert^2 - 2 \langle p, J_{B_1}v_n \rangle - 2 \gamma \langle p, A^* J_{B_2} (U-I) Av_n \rangle + \Vert v_n \Vert^2 \\
& \quad + \gamma^2 \Vert A^* \Vert^2 \Vert (U-I) Av_n \Vert^2 + 2 \langle J_{B_1}^* J_{B_1} v_n, \gamma A^* J_{B_2} (U-I)Av_n \rangle \\
& = \phi(p, v_n) + \gamma^2 \Vert A^* \Vert^2 \Vert (U-I) Av_n \Vert^2 + 2 \gamma \langle v_n - p, A^* J_{B_2} (U-I)Av_n \rangle . \numberthis \label{sp2}
\end{align*}
On the other hand, 
\begin{align*}
& \langle v_n-p, A^*J_{B_2}(U-I)Av_n \rangle \\
& = \langle Av_n-Ap, J_{B_2}(U-I)Av_n \rangle \\
& = - \langle (U-I)Av_n, J_{B_2}(U-I)Av_n \rangle+ \langle (U-I)Av_n+Av_n-Ap, J_{B_2}(U-I)Av_n \rangle \\
& = - \Vert (U-I)Av_n \Vert^2+\langle UAv_n-Ap, J_{B_2}(U-I)Av_n \rangle \\
& \leq  - \Vert (U-I)Av_n \Vert^2+\frac{1}{2} [ \Vert Av_n-Ap \Vert^2+\Vert (U-I)Av_n \Vert^2- \Vert U Av_n-Ap \Vert^2 ]\\
& = -\frac{1}{2} \Vert (U-I)Av_n \Vert^2. \numberthis \label{sp3}
\end{align*}
Hence, (\ref{sp2}) reduces to
\begin{align*}
\phi(p, u_{n+1}) & \leq \phi(p, v_n)+\gamma^2 L \Vert(U-I) Av_n \Vert^2 - \gamma \Vert (U-I)Av_n \Vert^2 \\
& = \phi(p, v_n) + \gamma (L \gamma -1) \Vert (U-I)Av_n \Vert^2. \numberthis \label{sp4}
\end{align*}
Since, $\gamma \in (0, \frac{1}{L})$, one has
\begin{align*}
\phi(p, u_{n+1}) \leq \phi(p, v_n) \leq \phi(p, u_n),~ \forall n \geq 0. \numberthis \label{sp5}
\end{align*}
This shows that $\{\phi(p, u_n)\}$ is a bounded and non-decreasing sequence and hence $\{u_n\}$ is also bounded and limit of $\phi(p, u_n)$ exists, $\forall n \geq 0$. This also implies that $\{v_n\}$ is bounded. \\
From (\ref{sp1}), we have
\begin{align*}
\alpha_n (1-\alpha_n) \psi(\Vert J_{B_1} u_n - J_{B_1} \Pi_M Tu_n \Vert) \leq  \phi(p, u_n) - \phi(p, v_n) - \alpha_n \eta(\phi(p, u_n). \numberthis \label{sp6}
\end{align*}
Since $\lim_{n \to \infty} \phi(p, u_n)$ exists and $\limsup_{n \to \infty} \alpha_n <1$, (\ref{sp6}) reduces to
\begin{align*}
\lim_{n \to \infty} \psi(\Vert J_{B_1} u_n - J_{B_1} \Pi_M Tu_n \Vert) =0,
\end{align*}
and by the properties of $\psi$ and uniform smoothness of $B_1$, we obtain that 
\begin{align*}
\lim_{n \to \infty} \Vert u_n - \Pi_M Tu_n \Vert =0. \numberthis \label{sp7}
\end{align*}
Next, from (\ref{sp4}), we deduce that
\begin{align*}
\gamma (L \gamma -1) \Vert (U-I)Av_n \Vert^2 & \leq \phi(p, v_n) - \phi(p, u_{n+1}) \leq \phi(p, u_n) - \phi(p, u_{n+1}).
\end{align*}
Hence, 
\begin{align*}
\lim_{n \to \infty} \Vert (U-I)Av_n \Vert =0. \numberthis \label{sp8}
\end{align*}
On the other hand, 
\begin{align*}
\phi(u_n, v_n) & = \phi(u_n, J_{B_1}^{-1}((1-\alpha_n) J_{B_1} u_n + \alpha_n  J_{B_1} \Pi_M Tu_n)) \\
& \leq \alpha_n \phi(u_n, u_n) +(1-\alpha_n) \phi(u_n, \Pi_M Tu_n).
\end{align*}
By using (\ref{sp7}), we obtain
\begin{align*}
\lim_{n \to \infty} \phi(u_n, v_n) =0, \numberthis \label{sp9}
\end{align*}
and so by Lemma \ref{imp1}, we have
\begin{align*}
\lim_{n \to \infty} \Vert u_n - v_n \Vert =0. \numberthis \label{sp10}
\end{align*}
Therefore, from (\ref{sp7}), we can conclude that 
\begin{align*}
\lim_{n \to \infty} \phi(u_n, Tu_n) & = \lim_{n \to \infty} [ \phi(u_n, \Pi_M Tu_n) + \phi(\Pi_M Tu_n, Tu_n) \\
& \quad + 2 \langle u_n - \Pi_M Tu_n, J \Pi_M Tu_n - J Tu_n \rangle ] \\
& = \text{dist}_\phi (M, N).
\end{align*}
This shows that $\{u_n\}$ is an (A$\phi$-BPS) with respect to the mapping $T$. \\
Since $\{u_n\}$ is bounded, there exists a weakly convergent subsequence, say $\{u_{n_k}\}$ of $\{u_n\}$ such that $u_{n_k} \xrightarrow{w} u^*,$ for some $u^* \in M$. Then, by (\ref{sp10}), $v_{n_k} \xrightarrow{w} u^*$ and so, $Av_{n_k} \xrightarrow{w} Au^*$. Since $T$ satisfies the $\phi$-proximal property and $\{u_n\}$ is (A$\phi$-BPS), so $\phi(u^*, Tu^*)=\text{dist}_\phi (M, N)$. i.e., $u^* \in \text{Best}_\phi^M (T)$. Also, using the fact that $U-I$ is demiclosed at zero, we conclude that $Au^* \in VIP(f, K)$. 
\end{proof}

\begin{remark}
Theorem \ref{weak} extends \cite[Theorem 3.1]{HA} in the setting of Hilbert spaces, where the convergence of sequences to best proximity points and solutions of monotone variational inclusion problems was proven under the assumption that the involved mappings are nonexpansive. In contrast, our result generalizes this framework to Banach spaces by considering proximally weakly suppressive mappings.
\end{remark}
\section{Concluding remarks}
  To conclude our work, we discussed the existence and convergence properties of $\phi$-best proximity points and characterized property ($\phi$-UC) in Banach spaces. We devised a shrinking projection scheme tailored for proving strong convergence theorem for the generated sequence of a non-self proximally weakly suppressive mapping. Furthermore, we proposed a projection method to address the split $\phi$-best proximity point problem and associated variational inequality problems. These contributions advance the theory of best proximity points and provide new tools for solving a broader class of nonlinear problems in Banach space settings.

\end{document}